\documentclass[11pt]{article}

\usepackage[margin = 1in]{geometry}
\usepackage{setspace}

\usepackage[T1]{fontenc}
\usepackage{charter} 
\usepackage{euler} 
\usepackage{float}
\usepackage{soul} 
\usepackage{bbm} 
\usepackage{verbatim}
\usepackage{amsmath}
\usepackage{amsthm}
\usepackage{amssymb}
\usepackage{textcomp}
\usepackage{graphicx}
\usepackage[english]{babel}
\usepackage{tikz}
\usepackage{dblfloatfix}
\usepackage{csquotes}
\usepackage{titlesec}

\usepackage{cancel}

\usepackage{enumerate}

\allowdisplaybreaks 

\usepackage{algorithm}
\usepackage{algpseudocode}

\newtheorem{theorem}{Theorem}[]

\newtheorem{lemma}[theorem]{Lemma}

\newtheorem{corollary}[theorem]{Corollary}

\usepackage{sectsty}
\sectionfont{\centering}

\newcommand{\1}{\mathbf{1}} 
\newcommand{\prob}{\mathbb{P}} 
\newcommand{\ex}{\mathbb{E}} 
\newcommand{\R}{\mathbbm{R}} 
\newcommand{\bin}{\ensuremath{\operatorname{binomial}}} 
\newcommand{\unif}{\ensuremath{\operatorname{uniform}}} 

\title{\bf{On the size of temporal cliques in subcritical random temporal graphs}
  \thanks{Luc Devroye acknowledges the support of NSERC grant A3450.
    G\'abor Lugosi acknowledges the support of Ayudas Fundación BBVA a
Proyectos de Investigación Científica 2021 and
the Spanish Ministry of Economy and Competitiveness grant PID2022-138268NB-I00, financed by MCIN/AEI/10.13039/501100011033,
FSE+MTM2015-67304-P, and FEDER, EU.).}
}

\author{
\bf{Caelan Atamanchuk} \\
Department of Mathematics and Statistics \\
McGill University \\
caelan.atamanchuk@gmail.com
\and
\bf{Luc Devroye} \\
School of Computer Science, McGill University, \\
  Montreal, Canada
lucdevroye@gmail.com
\and
\bf{G\'{a}bor Lugosi} \\
Department of Economics and Business, \\
Pompeu  Fabra University, Barcelona, Spain \\
ICREA, Pg. Lluís Companys 23, 08010 Barcelona, Spain \\
Barcelona Graduate School of Economics \\
gabor.lugosi@gmail.com
}

\begin{document}

\setstretch{1.2}

\maketitle

\begin{abstract}
  A \emph{random temporal graph} is an Erd\H{o}s-R\'enyi random graph $G(n,p)$, together with a random
  ordering of its edges. A path in the graph is called \emph{increasing} if the edges on the path appear
  in increasing order. A set $S$ of vertices forms a \emph{temporal clique} if for all $u,v \in S$, there is an increasing path from $u$ to $v$.
  Becker, Casteigts, Crescenzi, Kodric, Renken, Raskin and Zamaraev \cite{Becker2023} proved that if $p=c\log n/n$ for $c>1$, then,  with high probability, there
  is a temporal clique of size $n-o(n)$. On the other hand, for $c<1$,   with high probability,
  the largest temporal clique is of size $o(n)$.
  In this note, we improve the latter bound by showing that, for $c<1$, the largest temporal clique
  is of \emph{constant} size with high probability. 
\end{abstract}

\section{Introduction}\label{intro}

A \textit{temporal graph} $G = (V,E,\pi)$ is a graph $G=(V,E)$ together with an ordering $\pi:E \to \{1,\ldots,|E|\}$ on the edge set, interpreted as the times where the edges appear in the graph, often called the \textit{time stamps} of $G$. We say that an edge $e \in E$ \textit{precedes} an edge $e' \in E$ if $\pi(e) < \pi(e')$. A path from $u$ to $v$ is called $\textit{increasing}$ if each edge used in the path precedes the edge that is used after it, and we say that $v$ is \textit{reachable} from $u$ (or that $u$ can reach $v$) if an increasing path exists from $u$ to $v$. A set of vertices $S \subseteq V$ is called a \emph{temporal clique} if for all distinct vertices $u,v \in S$, there is an increasing path from $u$ to $v$ (and vice versa).


In this note we discuss temporal graphs where $\pi$ is a uniform permutation on the edges and $G$ is an Erd\H{o}s-R\'{e}nyi random graph. The resulting temporal graph is called a \textit{random simple temporal graph}; \textit{RSTG} for short.
Motivated by modelling time-dependent propagation processes,
this model was introduced by Casteigts, Raskin, Renken, and Zamaraev \cite{Casteigts2023}.

One may generate RSTGs by a simple method: start with the complete graph $K_n$, then assign each edge an independent $\unif(0,1)$ random variable $(U_e : e \in E)$ and delete every edge with $U_e > p$. In this construction, we say that $e$ precedes $e'$ if $U_e < U_{e'}$. We also call the labels $U_e$ the time stamps. Importantly, creating i.i.d.\ uniform time stamps like this allows us to extend the notion of a temporal graph to infinite graphs which is needed for our analysis.

Casteigts, Raskin, Renken, and Zamaraev \cite{Casteigts2023} studied connectivity properties of RSTGs. They
identified the thresholds for different strengths of connectivity to
be in the region where $p = \frac{c\log(n)}{n}$ for some constant $c >
0$
  (Throughout the paper, $\log$ denotes natural logarithm).
Furthering this work, Broutin, Kamčev and Lugosi \cite{Broutin2023} identified the asymptotic lengths of the longest and shortest increasing paths in RSTGs with high probability for values of $p$ in this range. (We say that an event $E=E(n)$ happens with high probability if $\prob(E) \to 1$ as $n \to \infty$). 
Becker, Casteigts, Crescenzi, Kodric, Renken, Raskin and Zamaraev  \cite{Becker2023} identified $p = \frac{\log(n)}{n}$ as the threshold for the appearance of large temporal cliques. In particular, they showed that for every $\epsilon >0$,
when $p \ge \frac{(1+\epsilon)\log(n)}{n}$, then
there is a temporal clique of size $n-o(n)$ with high probability, while  if
$p \le \frac{(1-\epsilon)\log(n)}{n}$, then every temporal clique is of size $o(n)$ with high probability.

RSTGs are a natural way to model time-dependent processes on networks like social interactions and infection spread. A closely related model is the \textit{random gossip protocol} model, in which a sequence of edges $e_1,\ldots,e_k$ are chosen uniformly from the edges of $K_n$ and constructed a graph $G_{n,k}$. Increasing paths are defined as for temporal graphs. Papers studying this model include Moon \cite{Moon1972}, Boyd and Steele \cite{Boyd1979}
and Haigh \cite{Haigh1981}.

For deterministic temporal graph models with random time stamps, see Chv\'{a}tal and Koml\'os \cite{Chvatal1971}, and Graham and Kleitman \cite{Graham1973}, 
Lavrov and Loh \cite{Lavrov2016} and Angel, Ferber, Sudakov and Tassion \cite{Angel2020}.

Our contribution is summarized in the following theorem.
It shows that in the subcritical regime $p= c\log n/n$ with $c<1$,
the size of the largest temporal clique is not only $o(n)$ but in fact, of size $O(1)$,
improving the upper bound of \cite{Becker2023}. This reveals a quite spectacular
phase transition around $p= \log n/n$, since for $c>1$, there is a temporal clique
of size $n-o(n)$. The behavior of the size of the largest temporal clique near the critical
regime remains an intriguing research problem.

\begin{theorem}\label{main}
Let $p = \frac{c\log(n)}{n}$, and let $G$ be an RSTG with edge probability $p$. If $c \in (0,1)$, then the largest temporal clique in $G$ is of size at most $\lceil \frac{1}{1-c}+1 \rceil$ with high probability.
\end{theorem}

Note that for $c \leq \frac{1}{2}$, Theorem \ref{main} asserts that $G$ has
no temporal clique of size $4$. This bound can't be improved, since 
for $p = \omega(\frac{1}{n})$ the static Erd\H{o}s-R\'{e}nyi graph contains a triangle with high probability. Moreover, every triangle is trivially a temporal clique of size $3$.
We conjecture that the upper bound of Theorem \ref{main} is sharp for all $c\in (0,1)$.

The proof of Theorem \ref{main} is based on relating the number of vertices that are reachable by monotone paths from a typical vertex to the total progeny of a certain ``temporal'' branching process.
We utilize the temporal branching process bounds to assert that the number of vertices that a collection of $m \geq \lceil 1 + \frac{1}{1-c}\rceil$ vertices can reach is small enough so that the chance of them forming a component unlikely enough that expected number of components of size $m$ tends to zero.

\section{Temporal branching processes}\label{temptree}

We begin this section by introducing temporal branching processes that are the key tool
in the proof of Theorem \ref{main}.
We only focus on branching processes with binomial offspring distribution as this is the degree distribution of a typical vertex in an Erd\H{o}s-R\'enyi random graph.
One way to generate these processes is as follows.
Start with an infinite rooted $n$-ary tree, add an independent $\unif(0,1)$ time stamp $U_e$
to every edge. Delete any edge with $U_e > p$. This decomposes the tree into a forest
and we only focus on the component that contains the root vertex.
We say that a vertex $v$ is reachable from the root if the unique path from the root to $v$ in the infinite $n$-ary tree has edge labels $U_e \leq p$ for each edge on the path and these labels are increasing on the path.
The subtree consisting of only vertices reachable from the root is a temporal branching process with a $\bin(n,p)$ offspring distribution. Throughout the rest of the paper, $T$ is always an infinite $n$-ary tree with such a labelling on the edges. The following sequence of results provides us with the necessary upper bounds for the size of the reachable set of $T$.

\begin{lemma}\label{treesep}
Let $P_1,\ldots,P_q$ be a finite collection of distinct infinite paths in $T$, and let $(X_k)_{k \geq 0}$ be a random walk down the tree, that is, $X_0$ is the root and $X_k$ is uniformly distributed over the children of $X_{k-1}$ for all $k \geq 1$. Then, $\prob(\tau \geq \ell) \leq \frac{q}{n^\ell}$, where
$$\tau = \max\{ k > 0 : X_0,\ldots,X_k \text{ coincides with one of the $q$ paths}\}.$$
\end{lemma}

\begin{proof}
If $\tau \geq \ell$, then $X_0,\ldots,X_\ell$ coincides with one of the $P_1,\ldots,P_q$. Since $X_k$ is uniform over the children of $X_{k-1}$ and the tree is $n$-ary, the result follows from the union bound yields.
\end{proof}

\begin{lemma}\label{treesurvival}
Let $T^*$ be the set of all vertices that are reachable from the root of $T$. If $v_1,\ldots,v_q$ are uniform vertices chosen from the $\ell$-th generation of $T$, then
$$\prob(v_1,\ldots,v_q \in T^*) \leq \frac{(q-1)!e^{npq}}{n^{\ell q}}.$$
Furthermore, when $\ell \geq (np)^4$ and $np \to \infty$ as $n \to \infty$,
$$\prob(v_1,\ldots,v_q \in T^*) = O\left(\frac{(q-1)!C^{q-1}}{n^{\ell q}}\right),$$
for some $C > 0$.
\end{lemma}

\begin{proof}
Suppose that $v_1,\ldots,v_r \in T^*$ and let $T'$ be a subtree of $T$ consisting of $r$ distinct infinite paths starting at the root through $v_1,\ldots,v_r$. Let $(X_k)_{k \geq 0}$ be a random walk down the tree (independent of $v_1,\ldots,v_r$) and let $\tau$ be as in Lemma \ref{treesep}. If $\tau = j$, then there are $\ell-j$ edges left that need to both exist and be increasing to have $X_\ell \in T^*$. Hence, if $V_r = \{v_1 , \ldots , v_r \in T^*\}$,
$$\prob(X_\ell \in T^* | V_r \cap \{\tau = j\}) = p^{\ell-j}\prob\big(\big\{\text{a path of length $\ell$ is increasing}\big\} \big|\big\{ \text{first $j$ edges are increasing}\big\}\big) = \frac{p^{\ell-j}j!}{\ell!}.$$
Combining this with Lemma \ref{treesep} yields
\begin{align}\label{sum}
\prob(X_\ell \in T^* | V_r)  &
\leq \sum_{j=0}^{\ell-1} \prob(\tau = j)\prob(X_\ell \in T^* | V_r \cap \{\tau = j\}) 
+ \prob(\tau \ge \ell)\prob(X_\ell \in T^* | V_r \cap \{\tau \ge \ell\}) 
 \nonumber \\
& \leq \sum_{j=0}^\ell \left(\frac{r}{n^j}\right)\frac{p^{\ell-j}j!}{\ell!} = \frac{r}{n^\ell}\sum_{j=0}^\ell\frac{(np)^{\ell-j}j!}{\ell!}.
\end{align}
To get the first bound we use the fact that $\frac{j!}{\ell!} \leq \frac{1}{(\ell-j)!}$ to get 
$$\prob(X_\ell \in T^*|V_r) \leq \frac{re^{np}}{n^\ell}.$$
Then, since $X_\ell$ is distributed uniformly across the $\ell$-th generation,  applying the above inequality repeatedly,
\begin{align*}
\prob(v_1,\ldots,v_q \in T^*) &= \prob(v_1 \in T^*) \cdot \prob(v_2 \in T^* | V_1) \cdots \prob(v_q \in T^* | V_{q-1}) \\
&\leq \frac{p^\ell}{\ell!}\prod_{r=1}^{q-1}\left(\frac{re^{np}}{n^\ell}\right) = \left(\frac{(np)^{\ell}}{\ell!}\right)\frac{(q-1)!e^{np(q-1)}}{n^{\ell q}} \leq \frac{(q-1)!e^{npq}}{n^{\ell q}}.
\end{align*} 
For the second inequality, we split the sum in (\ref{sum}) in two separate pieces
\begin{align*}
    A = \sum_{j=0}^{\ell - \sqrt{\ell}}\frac{(np)^{\ell - j}j!}{\ell!}, \quad \text{and} \quad B = \sum_{j=\ell -\sqrt{\ell}+1}^\ell \frac{(np)^{\ell-j}j!}{\ell!} = \sum_{k=0}^{\sqrt{\ell}-1} \frac{(np)^k}{\ell \cdot (\ell-1) \cdots (\ell-k+1)}.
\end{align*}
The first term may be bounded by
\begin{align*}
A &\leq \sum_{k=\sqrt{\ell}}^\ell \frac{(np)^k}{k!} \leq \left|e^{np}- \sum_{k=0}^{\sqrt{\ell}-1} \frac{(np)^k}{k!}\right| \leq \frac{e^{np}(np)^{\sqrt{\ell}}}{(\sqrt{\ell})!}~,
\end{align*}
where the second inequality follows from the Lagrange form of the remainder in Taylor's theorem.

For the second term, since $\ell \cdot (\ell-1) \cdots (\ell - k + 1) \geq \ell^{k}(1-\frac{1}{\sqrt{\ell}})^{\sqrt{\ell}}$ for all $0 \leq k \leq \sqrt{\ell}$, we have that
\begin{align*}
    B &\leq \frac{1}{\left(1-\frac{1}{\sqrt{\ell}}\right)^{\sqrt{\ell}}}\sum_{k=0}^{\sqrt{\ell}-1} \left(\frac{np}{\ell}\right)^k \leq \frac{1}{\left(1-\frac{np}{\ell}\right)\left(1-\frac{1}{\sqrt{\ell}}\right)^{\sqrt{\ell}}} = O(1),
\end{align*}
when $\ell \geq (np)^4$. Combining both the bounds along with Stirling's approximation, we conclude that there is some $C > 0$ such that
$$
\prob(X_\ell \in T^* | V_r) 
\leq \frac{r}{n^\ell}\left(\frac{e^{np}(np)^{\sqrt{\ell}}}{\sqrt{\ell}!} + O(1)\right) \leq \frac{r}{n^\ell}\left(\frac{e^{np}(np)^{(np)^2}e^{(np)^2}}{(np)^{2(np)^2}}+O(1)\right) \leq \frac{Cr}{n^\ell},$$
when $\ell \geq (np)^4$ and $np \to \infty$. Proceeding exactly as we did for the first inequality,
\begin{align*}
\prob(v_1,\ldots,v_q \in T^*) &= \prob(v_1 \in T^*) \cdot \prob(v_2 \in T^* | V_1) \cdots \prob(v_q \in T^* | V_{q-1}) \\
&\leq \frac{p^\ell}{\ell!}\frac{(q-1)!C^{q-1}}{n^{\ell (q-1)}} \leq \frac{(np)^{(np)^4}}{(np)^4!}\frac{(q-1)!C^{q-1}}{n^{\ell q}} = O\left( \frac{(q-1)!C^{q-1}}{n^{\ell q}}\right),
\end{align*} 
where in the final bound we use the fact that $\frac{x^{x^4}}{(x^4)!} = o(1)$ as $x\to \infty$, which is an immediate consequence of Stirling's approximation.
\end{proof}

We may use Lemma \ref{treesurvival} to bound the moments of the number of vertices reachable in a particular generation. For $\ell \geq 0$, denote by $Z_\ell$ the number of vertices in $T$ reachable from the root in the $\ell$-th generation. 

\begin{corollary}\label{reversejensen}
For all integers $\ell \ge 0$ and $q \ge 1$,
$\ex[Z_\ell^q] \leq (q -1)!e^{npq}$. Furthermore, when $\ell \geq (np)^4$ and $np \to \infty$, there is a constant $C > 0$ such that $\ex[Z_\ell^q] = O((q-1)!C^{q-1})$.
\end{corollary}

\begin{proof}
Denoting by $S_\ell$ the set of $n^{\ell}$ vertices in the $\ell$-th generation of $T$, we may write
$Z_\ell = \sum_{v\in S_{\ell}} \mathbf{1}_{\{v\in T^*\}}$. Then
\[
  \ex [Z_\ell^q] = n^{\ell q}\prob( v_1,\ldots,v_q \in T^*)~,
\]
where $v_1,\ldots,v_q$ are independent vertices chosen uniformly at random from $S_\ell$.
Combining this with Lemma \ref{treesurvival} implies the stated bounds.
\end{proof}

The next bound will control the number of vertices that a typical vertex in a simple random temporal graph can reach, further allowing us to control the size of temporal cliques.

\begin{theorem}\label{sizebound} Let $T^*$ be the set of vertices in $T$ that are reachable from the root and suppose that $np \to \infty$. Then, for any integer $q \geq 1$, there is a constant $c(q)$ such that
$$\ex[|T^*|^q] \leq c(q)(np)^{4q}e^{npq}.$$
\end{theorem}

\begin{proof}
Observe that
\begin{align*}
    \ex|T^*|^q & = \ex\left(\sum_{i=0}^{(np)^4}Z_i + \sum_{i > (np)^4} Z_i\right)^q  \\
& \le 2^{q-1} \ex\left(\sum_{i=0}^{(np)^4}Z_i \right)^q + 2^{q-1} \ex\left(\sum_{i > (np)^4} Z_i\right)^q
\quad \text{(by Jensen's inequality)} \\
&
\leq \underbrace{2^{q-1}((np)^4+1)^{q-1}\ex\left[ \sum_{i=0}^{(np)^4} Z_i^q \right]}_{:= I} + \underbrace{2^{q-1}q\sum_{t > 0}t^{q-1}\prob\left(\sum_{i > (np)^4} Z_i \geq t\right)}_{:= II}~,
\end{align*}
where in the last step we used Jensen's inequality to bound the first term and 
the identity $\ex[X^q] = \int qt^{q-1}\prob(X > t) dt$ to bound the second.
The first inequality of Corollary \ref{reversejensen} may be used to bound 
the expectation in term $I$,
 as 
$$\ex\left[ \sum_{i=0}^{(np)^4} Z_i^q\right] \leq 
((np)^{4}+1)
\left(\sup_{i \geq 0}\ex[Z_i^q]\right) \leq
((np)^{4}+1)
(q - 1)!e^{npq}.$$
To bound $II$, we may write
\begin{align*}
t^{q-1}\prob\left( \sum_{i > (np)^4} Z_i \geq t\right) &\leq t^{q-1}\prob\left( Z_{(np)^4 + \log t} > 0\right) + t^{q-1}\prob\left(\bigcup_{i=(np)^4 + 1}^{(np)^4 + \log t} \left\{Z_i \geq \frac{t}{\log t}\right\}\right) \\
&\leq \underbrace{t^{q-1}\prob\left( Z_{(np)^4 + \log t} > 0\right)}_{:=III} + \underbrace{t^{q-1}(\log t)\max_{(np)^4 < i \leq (np)^4 + \log t} \prob\left(Z_i \geq \frac{t}{\log t}\right)}_{:= IV},
\end{align*}
and we can bound the two terms separately. To bound $III$, note that at level $(np)^4+\log t$, there are $n^{(np)^4 + \log t}$ vertices, and they are each reachable with probability $p^{(np)^4 + \log t}/((np)^4+ \log t)!$. Thus, by Stirling's approximation,
\begin{align*}
III &\leq \frac{t^{q-1}(np)^{(np)^4 + \log t}}{((np)^4 + \log t)!} \\   
    &\leq \frac{t^{q-1}(enp)^{(np)^4 + \log t}}{(np)^{4(np)^4 + 4\log t}} \\
    &  = t^{q-1}\left( \frac{e}{(np)^3} \right)^{(np)^4 + \log t} \\
    &= t^{q-3\log(np)}\left(\frac{e}{(np)^3}\right)^{(np)^4}
\end{align*}
for any $t \geq 0$. In particular, this implies that $III$ is summable and converges to 0 when $np \to \infty$. Applying Markov's inequality and the second inequality in Corollary \ref{reversejensen} gives
$$IV \leq O\left( \frac{\log^{k+1} (t) ( k -1 )!
C^{k-1}
}{t^{k-q+1}}\right),$$
for any positive integer $k$ and $t \geq 0$. Choosing $k = q+1$ results in $IV$ being summable and bounded above by a constant depending only on $q$. Grouping up all that only depends on $q$ and upper bounding by some dominating constant $c(q)$ we get
$$\ex|T^*|^q \leq I + III + IV \leq c(q)(np)^{4q}e^{npq}.$$
\end{proof}

\section{Proof of Theorem \ref{main}}

We are now prepared to prove Theorem \ref{main}.
For labelled vertices $\{1,\ldots,m\}$ to form a temporal clique in an RSTG they need to all be reachable from one another. This can happen if and only if for all distinct $u,v \in \{1,\ldots,m\}$, there is a vertex $w$ that can reach $v$ with only edges that have time stamps above $p/2$, and is reachable from $u$ with only edges that have time stamps below $p/2$. With this in mind, for any $0\le a < b\le p$, we define $G_{[a,b]}$ to be the subgraph obtained from $G$ by only keeping edges with time stamps in $[a,b]$. Set $A_1,\ldots,A_m$ to be the collection of all vertices that are reachable from $1,\ldots,m$ in $G_{[0,p/2]}$ and $B_1,\ldots,B_m$ to be the collection of all vertices in $G_{[p/2,p]}$ that can reach $1,\ldots,m$. With this new notation, we can say that $\{1,\ldots,m\}$ form a temporal clique if for all $i,j \in \{1,\ldots,m\}$ distinct, the set $A_i \cap B_j$ is nonempty. 

It is important to note that $G_{[p/2,p]}$ and $G_{[0,p/2]}$ are identically distributed RSTGs, but are not independent. Observe that for any $0\le a < b \le p$, the RSTG $G_{[a,b]}$ is determined by the binary vector $X = (X_1,\ldots,X_{\binom{n}{2}})$
defined by $X_i = \1_{e_i \in G_{[a,b]}}$ (for some enumeration of the edges of $K_n$) and a random permutation $O_{[a,b]}$ of
$[\binom{n}{2}]$ that denotes the relative orderings of the edge labels. In the next lemma we consider certain functionals
of $G_{[a,b]}$, represented by $X$ and $O_{[a,b]}$. More precisely, such a functional is of the form
$f:\{0,1\}^{\binom{n}{2}}\times \mathrm{Sym}(\binom{n}{2}) \to \R$, where $\mathrm{Sym}(\binom{n}{2})$ is the set of permutations
of $[\binom{n}{2}]$.
The next lemma deals with the dependence between two subgraphs $G_{[a,b]}$ and $G_{[c,d]}$, for some $0 \leq a < b < c < d \leq p$.

\setcounter{theorem}{5}

\begin{lemma}\label{dependence}
Let $G$ be an RSTG with vertices labelled $\{1,\ldots,n\}$, and let $0 \leq a < b < c < d \leq p$. Set $X_i = \1_{e_i \in G_{[a,b]}}$, $Y_i = \1_{e_i \in G_{[c,d]}}$ for some enumeration of the edges of $K_n$, $e_1,\ldots,e_{\binom{n}{2}}$, and let $X = (X_1,\ldots,X_{\binom{n}{2}})$ and $Y = (Y_1,\ldots,Y_{\binom{n}{2}})$. Let $O_{[a,b]}$ and $O_{[c,d]}$ be the permutations that denote the relative orderings of the edges in the two graphs. 
Let
$f,g:\{0,1\}^{\binom{n}{2}}\times \mathrm{Sym}(\binom{n}{2}) \to \R$ be such that
 $f(x_1,\ldots,x_{\binom{n}{2}},s)$ and $g(x_1,\ldots,x_{\binom{n}{2}},s)$ are two non-decreasing functions in $x_1,\ldots,x_{\binom{n}{2}}$ for any fixed $s\in  \mathrm{Sym}(\binom{n}{2})$. 
Then
$$\ex\left[f\left(X,O_{[a,b]}\right)g\left(Y,O_{[c,d]}\right)\right] \leq \ex\left[f\left(X,O_{[a,b]}\right)\right] \ex\left[g\left(Y,O_{[c,d]}\right)\right].$$
In particular,
$$\ex\left[ |A_1|^q \cdots |A_m|^q \cdot |B_1|^q \cdots |B_m|^q \right] \leq \ex\left[|A_1|^q \cdots |A_m|^q \right]\ex\left[
|B_1|^q \cdots |B_m|^q\right] = \ex\left[|A_1|^q \cdots |A_m|^q\right]^2,$$
for all $q \geq 0$ and $A_1,B_1,\ldots,A_m,B_m$ defined as above.
\end{lemma}

\begin{proof}
Conditioned on $Z= (X_1,Y_1,\ldots,X_{\binom{n}{2}},Y_{\binom{n}{2}})$, all of the randomness of $f\left(X,O_{[a,b]}\right)$ and $g\left(Y,O_{[c,d]}\right)$ comes from the random relative orderings. Since $a < b < c < d$, the two random variables $O_{[a,b]}$ and $O_{[c,d]}$ are independent, which implies that $f\left(X,O_{[a,b]}\right)$ and $g\left(Y,O_{[c,d]}\right)$ must also be conditionally independent on $Z$. Hence, by the tower property of conditional expectation,
\begin{align*}
\ex\left[f\left(X,O_{[a,b]}\right)g\left(Y,O_{[c,d]}\right)\right] &= \ex\left[\ex\left[f\left(X,O_{[a,b]}\right)g\left(Y,O_{[c,d]}\right) \Big| Z \right]\right] \\
&= \ex\left[\ex\left[f\left(X,O_{[a,b]}\right) \Big| Z \right]\ex\left[g\left(Y,O_{[c,d]}\right) \Big| Z \right]\right]. \\
&= \ex\left[\ex\left[f\left(X,O_{[a,b]}\right) \Big| X \right]\ex\left[g\left(Y,O_{[c,d]}\right) \Big| Y\right]\right],
\end{align*}
where the final equality just follows from the fact that, once we condition on $X$, knowing $Y$ tells us nothing about $f(X,O_{[a,b]})$ and vice versa. The random variables
$$\ex\left[f\left(X,O_{[a,b]}\right) \Big| X = (z_1,\ldots,z_{\binom{n}{2}}) \right], \ \text{and} \ \ex\left[g\left(Y,O_{[c,d]}\right) \Big| Y = (z_1,\ldots,z_{\binom{n}{2}}) \right]$$
are non-decreasing functions in $z_1,\ldots,z_{\binom{n}{2}}$ by the definitions of $f$ and $g$. Furthermore, the collection of random variables $\{X_i : i \in \{1,\ldots,\binom{n}{2}\}\} \cup \{Y_i : i \in \{1,\ldots,\binom{n}{2}\}\}$ are negatively associated (this can be seen by combining Proposition 7 and Lemma 8 from Dubhashi and Ranjan \cite{Dubhashi1998}). With this, applying the tower property again gives
\begin{align*}
\ex\left[f\left(X,O_{[a,b]}\right)g\left(Y,O_{[c,d]}\right)\right]  \leq \ex\left[f\left(X,O_{[a,b]}\right)\right]\ex\left[g\left(Y,O_{[c,d]}\right)\right].   
\end{align*}
Observing that $|A_1| \cdots |A_m|$ and $|B_1| \cdots |B_m|$ satisfy the conditions of the first statement and are identically distributed is enough to complete the proof of the second inequality.
\end{proof}

The next lemma acts as a bridge between RSTGs and the temporal branching processes explored in the previous section. The idea behind the proof hinges on the fact that the sizes of $\bin(n,p)$ branching processes upper bounds the sizes of neighbourhoods around vertices in an Erd\H{o}s-R\'{e}nyi graph, though formalizing this idea takes some work. Equipped with this and Theorem \ref{sizebound}, the proof of Theorem \ref{main} is reduced to a routine use of the first-moment method.

\begin{lemma}\label{stoch}
$|A_1|$ is stochastically dominated by $|T^*|$, where $T^*$ is the set of vertices reachable from the root in a temporal branching process $T$ with offspring distribution $\bin(n,p/2)$. In particular, $\ex[|A_1|^q] \leq \ex[|T^*|^q]$ for all $q \geq 0$.
\end{lemma}

\begin{proof}
We can determine $A_1$ via the foremost tree algorithm from
Casteigts, Raskin, Renken, and Zamaraev \cite{Casteigts2023}.
  The algorithm builds a tree recursively, building a tree of
  increasing
  paths starting from an arbitrary
  vertex. The algorithm is defined  as follows:

\begin{itemize}
\item
    Initialize with $\tau_0 = 1$ and $G_0$ as the single vertex
    labelled $1$.
 \item
  While $\tau_k \leq p/2$, set $\tau_{k+1}$ to be the smallest time
stamp of an edge   connecting vertices in $G_k$ with vertices outside of $G_k$
that is
larger than $\tau_k$.
\item
  If $\tau_{k+1} \leq p/2$, add the corresponding edge $e_{k+1}$ and
vertex $v_{k+1}$ to obtain $G_{k+1}$.
\item
    If $\tau_k>p/2$, the algorithm terminates and outputs
$G_{k}$.
\end{itemize}

Note that $|A_1|$
  equals the number of vertices of the resulting tree $G_k$,
that is,
\begin{equation}\label{time}
|A_1| = \inf\left\{k \geq 0 : \tau_k > \frac{p}{2}\right\}~.
\end{equation}
This foremost tree algorithm can also be run on the tree $T$ as a way
to generate $T^*$ with the same procedure, and we denote the sequence
of timestamps in this graph as $\tau^*_k$. Additionally, we denote by
$E_k$ and $E^*_k$ the collection of all viable edges that could be
added during step $k$, that is, all edges from $G_k$ to $K_n$ that, if added, keep the graph $G_{k+1}$ as an increasing tree (all vertices reachable from the root). Note that by the definition of the algorithm, every edge in $E_k$ must have a time stamp that is at least $\tau_k$ and similarly for $E_k^*$ and $\tau^*_k$. Moreover, the time stamps of edges in $E_k$ and $E^*_k$ are uniformly distributed on $[\tau_{k-1},1]$ and $[\tau_{k-1}^*,1]$ respectively. By means of a direct inductive coupling we show that $\tau$ stochastically dominates $\tau^*$, $|E_k^*|$ stochastically dominates $|E_k|$, and hence $|T^*|$ must stochastically dominate $|A_1|$ by the characterization of (\ref{time}).

The base case of the induction is easy to see. By definition $|E_1| = n-1$, $|E_1^*| = n$,
  $\tau_1 \sim \min_{1 \leq i \leq n-1} U_{1,i}$, and $\tau_1^* \sim \min_{1 \leq n} U_{1,i}$. Thus, just using the same uniforms to generate both $\tau_1$ and $\tau_1^*$ is enough. Now suppose that there is some probability space $(\Omega_{k-1},\mathcal{F}_{k-1},\prob_{k-1})$ and random variables distributed as $|E_{k-1}|,|E_{k-1}^*|,\tau_{k-1},\tau^*_{k-1}$ (we just use the same symbols to denote these random variables) such that $|E_{k-1}|(\omega) \leq |E_{k-1}^*|(\omega)$ and $\tau^*_{k-1}(\omega) \leq \tau_{k-1}(\omega)$ for all $\omega \in \Omega_{k-1}$. In the $(k-1)$-th step of the algorithm we added a new vertex to both graphs, resulting in $n-k$ possible new edges to $G$ and $n$ edges to $T$ for the $k$-th step. Hence, since we cannot add edges that are below $\tau_{k-1}$ and $\tau_{k-1}^*$ respectively
$$|E_k| \sim |E_{k-1}| + \bin(n-k,1-\tau_{k-1}), \quad |E_k^*| \sim |E_{k-1}^*| + \bin(n,1-\tau_k^*),$$
and, recalling the distribution of time stamps in $E_k$ and $E_k^*$,
$$\tau_k \sim \tau_{k-1} + (1-\tau_{k-1})\min_{1 \leq i \leq |E_k|} U_{k,i}, \quad \tau_k^* \sim \tau^*_{k-1} + (1-\tau^*_{k-1})\min_{1 \leq i \leq |E_k^*|} U_{k,i}.$$
Let $(\Omega_k, \mathcal{F}_k,\prob_k)$ be the product of $(\Omega_{k-1},\mathcal{F}_{k-1},\prob_{k-1})$ with $(\Omega', \mathcal{F}',\prob')$, the probability space of $\binom{n}{2}+n$ independent uniform random variables, $(U_{k,i} : 1 \leq i \leq \binom{n}{2}+n)$. Here we can couple the binomial random variables by generating them as $\sum_{i=1}^k \1_{\{U_{k,i} \leq (1-\tau_{k-1})\}}$ and $\sum_{i=1}^k \1_{\{U_{k,i} \leq (1-\tau^*_{k-1}\}}$ respectively. Then, if we generate $|E_k|$ and $|E_k^*|$ with these binomials, by the inductive hypothesis, it must be the case that $|E_k|(\omega) \leq |E_k^*|(\omega)$ for all $\omega \in \Omega_k$. Similarly, using the uniforms $(U_{k,i} : n+1 \leq i \leq |E_k^*|)$ to generate both $\tau_k$ and $\tau_k^*$ according to their distributions results in also having $\tau_k^*(\omega) \leq \tau_k(\omega)$ for all $\omega \in \Omega_k$.
\end{proof}

With the lemmas out of the way we now prove our main result.

\begin{proof}[Proof of Theorem \ref{main}]
Let $m \geq 0$ and let $A_1,\ldots,A_m,B_1,\ldots,B_m$ be as defined in the beginning of this section. Let $N$ be the number of temporal cliques of size $m$ in $G$. Then, if we take $(v_{ij})_{i,j = 1}^m$ to be independently and uniformly chosen random vertices from the labelled set $\{1,\ldots,n\}$, we may apply Lemma $\ref{dependence}$ to get that
\begin{align*}
\ex[N] &\leq n^m\prob(\{1,\ldots,m\} \text{ form a temporal clique}) \\
&= n^{m+m(m-1)}\prob(v_{ij} \in A_i, v_{ij} \in B_j \ \forall i \ne j) \\
&= n^{m+m(m-1)}\ex\left[ \prob\Big(v_{ij} \in A_i, v_{ij} \in B_j \ \forall i \ne j \Big| |A_1|,\ldots,|A_m|,|B_1|,\ldots,|B_m| \Big)  \right] \\
&\leq n^{m+m(m-1)}\ex\left[ \left( \frac{|A_1|}{n} \right)^{m-1} \cdots \left( \frac{|A_m|}{n} \right)^{m-1} \left( \frac{|B_1|}{n} \right)^{m-1} \cdots \left( \frac{|B_m|}{n} \right)^{m-1} \right] \\
&\leq n^{m+m(m-1)}\ex\left[ \left( \frac{|A_1|}{n} \right)^{m-1} \cdots \left( \frac{|A_m|}{n} \right)^{m-1}\right]^2 \\
&\leq n^{m-m(m-1)}\ex[|A_1|^{m-1} \cdots |A_m|^{m-1}]^2.
\end{align*}
Applying H\"{o}lder's inequality along with Lemma \ref{stoch} and
Theorem \ref{sizebound}
  applied for the probability $p/2 = c log(n)/(2n)$,
gives us the upper bound 
\begin{align*}
  \ex[N]
  &
    \leq n^{m-m(m-1)}\ex[|A_1|^{m(m-1)}]^2 \leq
         \kappa_m \left(\frac{c}{2}\log(n)\right)^{8m(m-1)}n^{m-m(m-1)+c(m(m-1))}
\end{align*}
  where $\kappa_m$ is a constant depending on $m$ only.
If $m-m(m-1)+cm(m-1) < 0$, then $\ex[N] \to 0$ as $n \to \infty$. Rearranging, this inequality is equivalent to $m \geq \lceil \frac{1}{1-c}+1\rceil$ as $m$ is an integer.
\end{proof}

\bibliographystyle{plain}
\bibliography{ref}

\end{document}